\documentclass{amsart}

\usepackage{accents, amsfonts, amssymb, amsthm, amsmath, calc, cancel, caption, cite,color, enumitem, environ, etoolbox, eucal, graphics, graphicx, hyperref, latexsym, mathdots, multirow, nicematrix, pgfplots, subcaption, theoremref, tikz,tikz-cd, url}
\usetikzlibrary{fit}
\makeatletter
\newsavebox{\measure@tikzpicture}
\NewEnviron{scaletikzpicturetowidth}[1]{%
  \def\tikz@width{#1}%
  \def\tikzscale{1}\begin{lrbox}{\measure@tikzpicture}%
  \BODY
  \end{lrbox}%
  \pgfmathparse{#1/\wd\measure@tikzpicture}%
  \edef\tikzscale{\pgfmathresult}%
  \BODY
}
\makeatother

\AtEndEnvironment{proof}{\setcounter{claim}{0}}

\numberwithin{equation}{section}

\theoremstyle{plain}
\newtheorem{theorem}{Theorem}[section]
\newtheorem{lemma}[theorem]{Lemma}

\theoremstyle{definition}
\newtheorem{definition}[theorem]{Definition}

\newtheorem{example}[theorem]{Example}

\newtheorem{claim}{Claim}

\renewcommand{\phi}{\varphi}
\renewcommand{\epsilon}{\varepsilon}
\newcommand{\mb}[1]{\mathbf{#1}}
\renewcommand{\vec}[1]{\mathbf{#1}}

\newcommand{\mc}[1]{\mathcal{#1}}

\newcommand{\A}{\mb{A}}

\newcommand{\B}{\mb{B}}
\newcommand{\C}{\mb{C}}

\newcommand{\M}{\mb{M}}

\newcommand{\U}{\mb{U}}

\newcommand{\x}{\vec{x}}

\newcommand{\2}{\mb{2}}

\newcommand{\PI}{\mb{3}^\text{p}}
\newcommand{\TC}{\mb{3}^\text{c}}

\newcommand{\bdot}{\boldsymbol{\cdot}}

\newcommand{\meet}{\wedge}
\newcommand{\join}{\vee}

\DeclareMathOperator{\ds}{ds}
\DeclareMathOperator{\cd}{cd}

\title{Commuting degrees for BCK-algebras}

\title{Commuting degrees for BCK-algebras}
\author{C. Matthew Evans}
\date{}

\begin{document}

\begin{abstract}
We discuss the following question: given a finite BCK-algebra, what is the probability that two randomly selected elements commute? We call this probability the \textit{commuting degree} of a BCK-algebra. In a previous paper, the author gave sharp upper and lower bounds for the commuting degree of a BCK-algebra with order $n$. We expand on those results in this paper: we show that, for each $n\geq 3$, there is a BCK-algebra of order $n$ realizing each possible commuting degree and that the minimum commuting degree is achieved by a unique BCK-algebra of order $n$ Additionally, we show that every rational number in $(0,1]$ is the commuting degree of some finite BCK-algebra.
\end{abstract}

\maketitle

\section{Introduction}

In 1973, Gustafson proved that the probability of two elements commuting in a finite non-Abelian group $G$ is at most $\frac{5}{8}$ \cite{gustafson73}. There are several ways to generalize this idea, but many such generalizations fall under the umbrella of the following question: what is the probability that a uniformly randomly selected $k$-tuple over a finite algebraic structure satisfies some specified first-order formula in $k$ free variables? We make this precise as follows:

\begin{definition}\label{defn:ds} Given a first-order language $\mathcal{L}$, a finite $\mathcal{L}$-structure $M$, and an $\mathcal{L}$-formula $\phi(x_1, x_2, \ldots x_k)$ in $k$ free variables, the quantity
\[\ds(\phi, M)=\frac{|\,\{\,(a_1, a_2,\ldots, a_k)\in M^k\mid \phi(a_1, a_2, \ldots a_k)\,\}|}{|M|^k}\] is the \emph{degree of satisfiability} of the formula $\phi$ in the structure $M$.
\end{definition}

\begin{definition} Let $T$ be a theory over a first-order language $\mathcal{L}$ and $\phi$ an $\mathcal{L}$-formula in $k$ free variables. We say that $\phi$ has \emph{finite satisfiability gap} $\epsilon$ in $T$ if there is a constant $\epsilon >0$ such that, for every finite model $M$ of $T$, either $\ds(\phi, M)=1$ or $\ds(\phi, M)\leq 1-\epsilon$\,.
\end{definition}

Using Gustasfon's result as an example, in the language of groups, the commutativity equation $xy=yx$ has finite satisfiability gap $\frac{3}{8}$: every finite group either has degree of satisfiability 1 (if it is Abelian) or no larger than $\frac{5}{8}$.

There is a substantial literature about commuting probabilities for finite groups; rather than list a lengthy list of citations here, we refer the reader to two specific papers and the bibliographies contained therein. The first is a 2006 paper by Guralnick and Robinson \cite{GR06} which appears to capture the zeitgeist of commuting probabilities in groups up to the early 2000's, and the second is the survey paper \cite{DNP13} by Das, Nath, and Pournaki from 2013.

Probabilities for some commutator-like equations in groups were considered by Delizia, Jezernik, Moravec, and Nicotera in \cite{DJM20} as well as by Lescot \cite{lescot95}, while Kocsis \cite{kocsis20} studied probabilities for various other equations generalizing the commutativity equation. MacHale investigated commuting probability for finite rings in \cite{machale76}, and there is more recent work focused on finite rings by Buckley and MacHale \cite{BM13, BM17}, Buckley, MacHale, and N\'{i} Sh\'{e} \cite{BMS13}, and Basnet, Dutta, and Nath \cite{BDN17}. Other probability questions in finite rings have also been explored; for example, the probability that two randomly chosen elements multiply to 0 has been examined in \cite{dolzan22, EJ18, SV24}, while the probability that two randomly chosen elements multiply to some arbitrary element $x$ has been examined in \cite{RS23, RBH19} and a very recent pre-print by Dol\v{z}an \cite{dolzan25}.

A recent paper of Bumpus and Kocsis \cite{BK24} considers probability questions for the class of Heyting algebras, which are the algebraic semantics for intuitionistic logic. They showed that, among all one-variable equations in the language of Heyting algebras, only three have a finite satisfiability gap: $x=\top$, $\neg x=\top$, and $x\join\neg x=\top$. Inspired by their work, the present author investigated satisfiability gaps for some equations in the language of BCK-algebras, a class of algebraic structures introduced by Imai and Is\'{e}ki \cite{II66} as the algebraic semantics for a non-classical logic having only implication. In \cite{evans23}, the author gave sharp bounds for the probability that two elements in a finite non-commutative BCK-algebra commute -- what we call the \emph{commuting degree} of the algebra -- by finding algebras that realized those bounds. As a consequence, the commutativity equation for BCK-algebras does not have a finite satisfiability gap.

Empirically, it was observed that for BCK-algebras of orders $n=3$, $4$, and $5$, every possible commuting degree is obtained by some algebra of that order. It was conjectured that this is true for all $n\geq 3$, and further, that every rational in $(0,1]$ occurs as the commuting degree of some BCK-algebra. This conjecture indicates that BCK-algebras are somewhat like semigroups: Ponomarenko and Selinski proved in \cite{PS12} that every rational in $(0,1]$ is the commuting probability for some semigroup.

The aim of this paper is to prove the two conjectures above. Additionally, we show that for each $n$, there is exactly one algebra (up to isomorphism) that achieves the minimum commuting degree. In the next section, we define the class of BCK-algebras and discuss two techniques for constructing new BCK-algebras. In section 3, we discuss the relevant material from \cite{evans23} and prove the main results.

\section{Preliminaries}

\begin{definition}A \emph{BCK-algebra} is an algebra $\A=\langle A; \bdot, 0\rangle$ of type $(2,0)$ such that 
\begin{enumerate}
\item[(BCK1)] $\bigl[(x\bdot y)\bdot(x\bdot z)\bigr]\bdot(z\bdot y)=0$
\item[(BCK2)] $\bigl[x\bdot (x\bdot y)\bigr]\bdot y=0$
\item[(BCK3)] $x\bdot x=0$
\item[(BCK4)] $0\bdot x=0$
\item[(BCK5)] $x\bdot y=0$ and $y\bdot x=0$ imply $x=y$.
\end{enumerate} for all $x,y,z\in A$.
\end{definition}

These algebras are partially ordered by: $x\leq y$ if and only if $x\bdot y=0$. Note that 0 is the least element by (BCK4). One can also show that $x\bdot 0=x$ for all $x\in \A$. We say the algebra $\A$ is \emph{bounded} if there exists an element $\top$ such that $x\bdot\top =0$ for all $x\in \A$. By (BCK5), if $\A$ contains incomparable elements $a$ and $b$, which we indicate with the notation $a\perp b$, then we must have either $a\bdot b\neq 0$ or $b\bdot a\neq 0$.

Define a term operation $\meet$ by $x\meet y:=y\bdot(y\bdot x)$. The element $x\meet y$ is a lower bound for $x$ and $y$, but in general is not the greatest lower bound of $x$ and $y$. One can show that if $x\meet y=y\meet x$, then $x\meet y$ is the greatest lower bound, and in this case we say that $x$ and $y$ \emph{commute}. If this equation holds for all $x,y\in\A$, we say the algebra $\A$ is \emph{commutative}. The class of all commutative BCK-algebras is an important and well-studied subclass of all BCK-algebras. While the class of all BCK-algebras is not a variety, the class of commutative BCK-algebras do form a variety (that is, this class is closed under taking subalgebras, products, and homomorphic images); for more detail on the elementary properties of BCK-algebras we refer the reader to \cite{it76, it78, mj94}.

Here we give two methods for constructing new BCK-algebras. Let $\{\A_\lambda\}_{\lambda\in\Lambda}$ be a family of BCK-algebras such that $A_\lambda\cap A_\mu=\{0\}$ for $\lambda\neq \mu$, and let $U$ denote the union of the $A_\lambda$'s. Equip $U$ with the operation 
\[x\bdot y=\begin{cases}x\bdot_\lambda y &\text{if $x,y\in A_\lambda$}\\x&\text{otherwise}\end{cases}\,.\] Then $U$ is a BCK-algebra we will call the \emph{BCK-union} of the $\A_\lambda$'s and denote it by $\U=\bigsqcup_{\lambda\in\Lambda}\A_\lambda$. This construction first appears in \cite{it76}, though only in the case $|\Lambda|=2$. For a full proof that this is a BCK-algebra, see \cite{evans20}.

Next, given any BCK-algebra $\A$ of order $n-1$, we construct a new BCK-algebra of order $n$ by appending a new top element, call it $\top$, and extending the BCK-operation as follows:
\begin{align*}
x\bdot \top&=0\\
\top\bdot \top&=0\\
\top\bdot x&=\top
\end{align*} for all $x\in\A$. This is known as \emph{Is\'{e}ki's extension of $\A$}, which we will denote $\A\oplus\top$. Is\'{e}ki's extension always yields a bounded non-commutative BCK-algebra having $\A$ as a maximal subalgebra \cite{iseki75}. We also note an important fact here: every BCK-algebra of order $n$ contains a subalgebra of order $n-1$ \cite{hao86}.

We call attention to three BCK-algebras of small order that will be used in conjunction with the constructions above. We give their Cayley tables in Table \ref{tab:algs}.
\begin{table}[h]
{\centering
\begin{tabular}{c||cc}
$\bdot$  & 0              & 1 \\\hline\hline
0             & 0              & 0 \\
1             & 1              & 0
\end{tabular}
\hspace{1cm}
\begin{tabular}{c||ccc}
$\bdot$  & 0              & $1$ & $2$ \\\hline\hline
0             & 0              & 0 & 0 \\
$1$             & $1$              & 0 & 0 \\
$2$             & $2$              & $2$ & 0 
\end{tabular}
\hspace{1cm}
\begin{tabular}{c||ccc}
$\bdot$  & 0              & $1$ & $2$ \\\hline\hline
0             & 0              & 0 & 0 \\
$1$             & $1$              & 0 & 0 \\
$2$             & $2$              & $1$ & 0 
\end{tabular}
\caption{\label{tab:algs} The algebras $\2$, $\PI$, and $\TC$}
}
\end{table} The algebra $\2$ is the unique (up to isomorphism) BCK-algebra of order 2, and it is commutative. The algebra $\PI$ is not commutative, but it is \emph{positive implicative}, meaning it satisfies the equation $x\bdot y=(x\bdot y)\bdot y$ for all $x,y\in\PI$. The algebra $\TC$ is commutative but not positive implicative. The posets for all three algebras above are chains, and we note that $\PI\cong \2\oplus\top$.

\section{Commuting degrees}\label{sec:cd}

Given a BCK-algebra $\A$, let $C(\A)$ denote the set of commuting pairs of elements:\[C(\A)=\{\, (x,y)\in A^2\,\mid\, x\meet y= y\meet x\,\}\,.\] Following definition \ref{defn:ds}, define the \emph{commuting degree} of $\A$, denoted $\cd(\A)$, to be the degree of satisfiability of the commutativity equation $x\meet y=y\meet x$; that is, \[\cd(\A)=\frac{|C(\A)|}{|\A|^2}\,.\] This is the probability that two elements chosen uniformly randomly (with replacement) commute with one another.

In any BCK-algebra $\A$, we have $x\meet 0=0\meet x$ for all $x\in \A$, and every element commutes with itself by (BCK3). So every pair of the form $(x,0)$, $(0,x)$, and $(x,x)$ is in $C(\A)$. Thus, if $|\A|=n$, we have $|C(\A)| \geq 3n-2$ and \[\cd(\A)\geq \frac{3n-2}{n^2}\,.\] Of course, if $\A$ is commutative then $\cd(\A)=1$, but if $\A$ is non-commutative we have
\[\frac{3n-2}{n^2}\leq \cd(\A)\leq \frac{n^2-2}{n^2}\] and these bounds are sharp.

Let $\U=\bigsqcup_{\lambda\in\Lambda}\A_\lambda$ be a BCK-union of a family of BCK-algebras. If two elements $a$ and $b$ lie in distinct components $A_\lambda$ and $A_\mu$ ($\lambda\neq \mu)$, then $a$ and $b$ necessarily commute:
\begin{align*}
a\meet b &= b\bdot (b\bdot a)=b\bdot b=0\\
b\meet a &= a\bdot (a\bdot b)=a\bdot a=0\,.
\end{align*} Heuristically, then, we expect that the commuting degree is larger when an algebra has many incomparable elements, and smaller when an algebra has many comparable elements. The following four results can be found in \cite{evans23}, though we present sketches for Theorems \ref{max cd} and \ref{min cd}.


\begin{lemma}\label{union with 2}
Let $\A$ be a BCK-algebra with $|\A|=n$ and $\cd(\A)=\frac{k}{n^2}$. Then 
\[\cd(\A\sqcup\2)=\frac{k+2n+1}{(n+1)^2}\,.\]
\end{lemma}

\begin{theorem}\label{max cd}
For each $n\geq 3$, there is a non-commutative BCK-algebra of order $n$ realizing the maximum commuting degree $\frac{n^2-2}{n^2}$.
\end{theorem}

\begin{proof}[Sketch of proof.]
For $n=3$, the algebra $\PI$ defined in Table \ref{tab:algs} has commuting degree $\frac{7}{9}$. Then we define a family of algebras $\B_n$ by
\begin{align*}
\B_3&=\PI\\ 
\B_n&=\B_{n-1}\sqcup\2 \text{ for $n>3$\,.}
\end{align*}
and by inducation one can show $\cd(\B_n)=\frac{n^2-2}{n^2}$.
\end{proof} If we label the atoms of $\B_n$ by $\{a_i\}_{i=1}^{n-2}$, then Figure \ref{fig:B_n} shows the Hasse diagram, verifying our intuition that commuting degree is large when the algebra has many incomparable elements.
\begin{figure}[h]
\centering
\begin{tikzpicture}
\filldraw (0,0) circle (2pt);
\filldraw (-2,1) circle (2pt);
\filldraw (-2,2) circle (2pt);
\filldraw (-1,1) circle (2pt);
\filldraw (0,1) circle (2pt);
\filldraw (2,1) circle (2pt);
\draw [-] (0,0) -- (-2,1) -- (-2,2);
\draw [-] (0,0) -- (-1,1);
\draw [-] (0,0) -- (0,1);
\draw [-] (0,0) -- (2,1);
	\node at (0,-.4) {\small 0};
	\node at (-2.4, 1) {\small $a_1$};
	\node at (-1, 1.3) {\small $a_2$};
	\node at (0, 1.3) {\small $a_3$};
\filldraw (.6,1) circle (.7pt);
\filldraw (1,1) circle (.7pt);
\filldraw (1.4,1) circle (.7pt);
	\node at (2, 1.3) {\small $a_{n-2}$};
\end{tikzpicture}
\caption{Hasse diagram for $\B_n$}
\label{fig:B_n}
\end{figure}
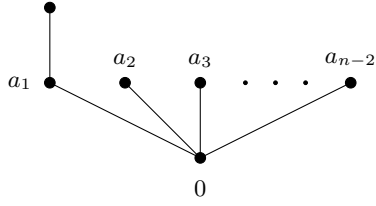

We note that, aside from $n=3$, there are several algebras of order $n$ that achieve this maximum commuting degree. For example, for $n=4$, there are three distinct algebras with commuting degree $\tfrac{14}{16}$, while for $n=5$ there are eight distinct algebras with commuting degree $\tfrac{23}{25}$.

We next consider the behavior of commuting degree with respective to Is\'{e}ki's extension.

\begin{lemma}\label{cd of sum}
Suppose $\A$ is a BCK-algebra with $|\A|=n$ and $\cd(\A)=\frac{k}{n^2}$. Then 
\[\cd(\A\oplus\top)=\frac{k+3}{(n+1)^2}\,.\]
\end{lemma}

The idea of the proof is simple: the new element $\top$ does not commute with any non-zero element of $\A$, but it does commute with 0 and itself, meaning \[(0,\top), (\top, 0), (\top, \top)\in C(\A\oplus\top)\] and therefore $|C(\A\oplus\top)|=k+3$.

\begin{theorem}\label{min cd}
For each $n\geq 3$, there is a BCK-algebra of order $n$ realizing the minimum commuting degree $\frac{3n-2}{n^2}$.
\end{theorem}

\begin{proof}[Sketch of proof.]
For $n=3$, the algebra $\PI$ has commuting degree $\frac{7}{9}$. Define a new family of BCK-algebras as follows:
\begin{align*}
\M_3 &= \PI \text{ and}\\ 
\M_n&=\M_{n-1}\oplus\top\text{ for $n>3$\,.}
\end{align*} One can show that $\cd(\M_n)=\frac{3n-2}{n^2}$\,.
\end{proof} Since the $\M_n$'s are linear, the above result verifies our intuition that commuting degree is small when there are many comparable elements. 

Unlike for the maximum commuting degree, it turns out for each $n$ there is a unique algebra achieving the minimum commuting degree, that being the algebra $\M_n$ defined above. (In fact, if we extend our definition of $\M_n$ to include $\M_2=\2$, then this is also true for $n=2$ since there is one possible value for commuting degree, 1, and only one BCK-algebra of order 2 which is necessarily commutative.)

\begin{theorem}\label{unique minimum}
For each $n\geq 3$, the algebra $\M_n$ is the unique BCK-algebra of order $n$ with commuting degree $\frac{3n-2}{n^2}$.
\end{theorem}

\begin{proof}
Knowing the statement is true for $n=3$, we induct on the order $n$. Assume the statement is true for some $n>3$. That is, assume $\M_n$ is the unique algebra of order $n$ that achieves the minimum commuting degree $\frac{3n-2}{n^2}$.

Let $\A$ be an algebra of order $n+1$ and suppose $\A$ achieves the minimum commuting degree $\frac{3(n+1)-2}{(n+1)^2}$, meaning $|C(\A)|=3n+1$. We know that $\A$ is a one-element extension of a subalgebra $\B$ of order $n$; write $A=B\cup\{a\}$ with $a\notin B$. If $B\not\cong\M_n$, then $|C(\B)|> 3n-2$. But we know that $(a,0), (0,a), (a,a)\in C(\A)$, and so
\[|C(\A)|> (3n-2)+3=3n+1\,,\] a contradiction. Thus, $B\cong\M_n$, and in the rest of the proof we identify the subalgebra $\B$ with $\M_n$. Therefore, the Cayley table for $\A$ has the form shown in Table \ref{tab:alg A}, and our goal is to show $\A\cong \M_{n+1}$. We recall that $x\leq y$ if and only if $x\bdot y=0$, and that $x\meet y$ is a common lower bound for $x$ and $y$.
\begin{table}[h]
{\centering
$\begin{NiceArray}{c||ccccccc}
\bdot  & 0  & 1 & 2 & \cdots & n-2 & n-1 & a \\\hline\hline
0           &0&0&0&\cdots&0&0&0\\
1           &1&0&0&\cdots&0&0& \\
2           &2&2&0&\cdots&0&0&\\
\vdots &\vdots&\vdots&\vdots&&\vdots&\vdots&\\
n-2 &n-2&n-2&n-2&\cdots&0&0&\\
n-1 &n-1&n-1&n-1&\cdots&n-1&0&\\
a &a&&&&&&0
\CodeAfter
	\tikz \draw (3-|8) rectangle (8-|9);
	\tikz \draw (8-|3) rectangle (9-|8);
\end{NiceArray}$
\caption{\label{tab:alg A} The algebra $\A$}
}\end{table}

{\renewcommand{\qedsymbol}{$\blacksquare$}
\begin{claim}The element $a$ must be comparable to every element of $\M_n$, meaning the Hasse diagram for $\A$ must be linear.
\end{claim}

\begin{proof}
Assume for contradiction that $a\perp k+1$, and suppose $k+1$ is the smallest element of $\M_n$ incomparable to $a$. So the Hasse diagram for $\A$ looks like Figure \ref{subfig1}.
\begin{figure}[ht]
\begin{subfigure}{0.45\linewidth}
\centering
\begin{tikzpicture}
\filldraw (0,0) circle (2pt);
\filldraw (0,1) circle (2pt);
\filldraw (0,2) circle (2pt);
\filldraw (0,3) circle (2pt);
\filldraw (1,3) circle (2pt);
\filldraw (0,4) circle (2pt);
\filldraw (0,5) circle (2pt);
\draw [-] (0,0) -- (0,1);
\draw [-] (0,2) -- (0,3);
\draw [-] (0,2) -- (1,3);
\draw [-] (0,4) -- (0,5);
	\node at (0,-.4) {\small 0};
	\node at (-.35, 1) {\small $1$};
	\node at (-.35, 2) {\small $k$};
	\node at (-.55, 3) {\small $k+1$};
	\node at (1.3, 3) {\small $a$};
	\node at (-.55, 4) {\small $n-2$};
	\node at (-.55, 5) {\small $n-1$};
\filldraw (0,1.25) circle (.7pt);
\filldraw (0,1.5) circle (.7pt);
\filldraw (0,1.75) circle (.7pt);
\filldraw (0,3.25) circle (.7pt);
\filldraw (0,3.5) circle (.7pt);
\filldraw (0,3.75) circle (.7pt);
\filldraw (.75,3.25) circle (.7pt);
\filldraw (.5,3.5) circle (.7pt);
\filldraw (.25,3.75) circle (.7pt);
\end{tikzpicture}
\caption{$\A$ if $a$ is incomparable to $k+1$}
    \label{subfig1}
    \end{subfigure}\hfill
\begin{subfigure}{0.45\linewidth}
\begin{tikzpicture}
\filldraw (0,0) circle (2pt);
\filldraw (0,1) circle (2pt);
\filldraw (0,2) circle (2pt);
\filldraw (0,3) circle (2pt);
\filldraw (0,4) circle (2pt);
\filldraw (0,5) circle (2pt);
\filldraw (0,6) circle (2pt);
\draw [-] (0,0) -- (0,1);
\draw [-] (0,2) -- (0,4);
\draw [-] (0,5) -- (0,6);
	\node at (0,-.4) {\small 0};
	\node at (-.35, 1) {\small $1$};
	\node at (-.35, 2) {\small $k$};
	\node at (-.35, 3) {\small $a$};
	\node at (-.55, 4) {\small $k+1$};
	\node at (-.55, 5) {\small $n-2$};
	\node at (-.55, 6) {\small $n-1$};
\filldraw (0,1.25) circle (.7pt);
\filldraw (0,1.5) circle (.7pt);
\filldraw (0,1.75) circle (.7pt);
\filldraw (0,4.25) circle (.7pt);
\filldraw (0,4.5) circle (.7pt);
\filldraw (0,4.75) circle (.7pt);
\end{tikzpicture}
\caption{$\A$ if $a$ is comparable to all elements}
    \label{subfig2}
\end{subfigure}
\caption{Possible Hasse diagrams for $\A$}
\end{figure}
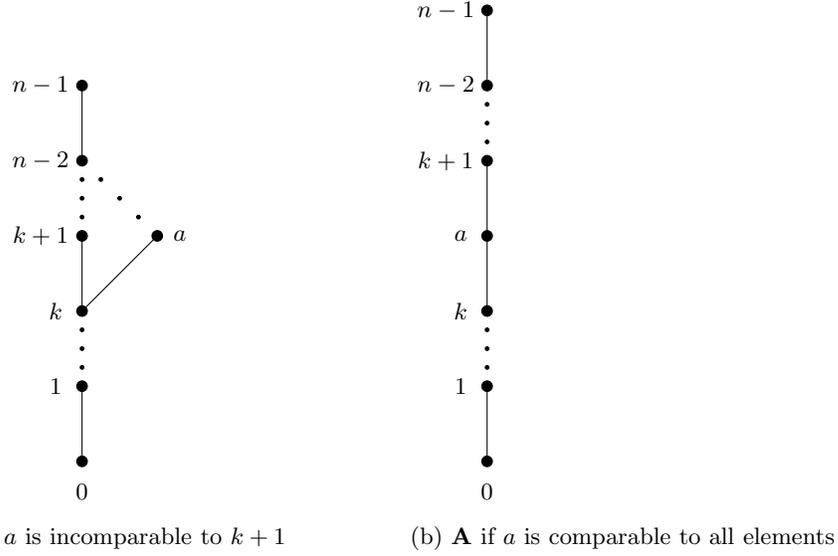

For $j\in\{0,1,\ldots, k\}$, we have $j\leq a$, meaning $j\bdot a=0$ and $a\meet j=j\bdot(j\bdot a)=j$. To avoid contradicting the minimality of $\A$ with respect to commuting degree, we need $j\meet a\in[0, j-1]$. 

{\renewcommand{\qedsymbol}{$\blacksquare$}
\begin{claim}We must have $a\bdot j=a$ for all $j\in\{0,1,\ldots, k\}$.
\end{claim}

\begin{proof}
For $j=1$, the preceding paragraph implies $1\meet a=a\bdot(a\bdot1)=0$. We know $a\bdot 1\neq 0$ since $a\not\leq 1$, but also $a\bdot 1\leq a$, and so $a\bdot 1\in\{1,2,\ldots, k,a\}$. Say $a\bdot 1=i\in\{1,2,\ldots, k\}$. Then \[a\bdot i=a\bdot(a\bdot 1)=0\] which implies $a\leq i$, a contradiction. Thus, we have $a\bdot 1=a$.

Assume we have proven $a\bdot i=a$ for all $i\in\{1,\ldots, j-1\}\subseteq[0,a]$. We will show that $a\bdot j=a$ as well. 

Suppose $j\meet a\neq 0$, say $j\meet a=i$ for some $1\leq i\leq j-1$. Then
\[i=j\meet a=a\bdot(a\bdot j)=(a\bdot i)\bdot(a\bdot j)=\bigl(a\bdot (a\bdot j)\bigr)\bdot i=i\bdot i=0\,,\] a contradiction. Thus, we must have $j\meet a=0$. 

We know $0\neq a\bdot j\leq a$, so $a\bdot j\in\{1, 2,\ldots, k,a\}$. If $a\bdot j=i$ with $1\leq i\leq j$, then
\[j\meet a=a\bdot (a\bdot j)=a\bdot i=\begin{cases}a&\text{if $i<j$}\\j&\text{if $i=j$}\end{cases}\,,\] a contradiction in either case. If $a\bdot j=\ell$ with $j+1\leq \ell\leq k$, then
\[a\bdot\ell=a\bdot(a\bdot j)=j\meet a=0\] which implies $a\leq \ell$, a contradiction. Thus, we must have $a\bdot j=a$. By induction, this means that $a\bdot j=a$ for all $j\in\{0,1,2,\ldots, k\}$.
\end{proof}}

Now, since $a\perp k+1$, we know $a\bdot (k+1)\neq 0$ and $(k+1)\bdot a\neq 0$. Since $a\bdot (k+1)\leq a$, we have $a\bdot (k+1)\in\{1,2,\ldots, k,a\}$. If $a\bdot (k+1)=i\in\{1,\ldots, k\}$, then 
\[(k+1)\meet a=a\bdot\bigl(a\bdot(k+1)\bigr)=a\bdot i=a\] by Claim 2, but this is a contradiction since $(k+1)\meet a<a$. Thus, $a\bdot (k+1)=a$ and therefore $(k+1)\meet a=0$.

Next, $a\meet (k+1)\in\{1,\ldots, k\}$ and $0\neq (k+1)\bdot a\leq k+1$. If $(k+1)\bdot a=i\in\{1,\ldots, k\}$, then 
\[a\meet (k+1)=(k+1)\bdot\bigl((k+1)\bdot a\bigr)=(k+1)\bdot i=k+1\notin\{1,\ldots, k\}\,,\] a contradiction. But also, if $(k+1)\bdot a=k+1$, then \[a\meet (k+1)=0=(k+1)\meet a\,,\] which contradicts the minimality of $\A$.

Hence, the new element $a$ must be comparable to all elements of the subalgebra $\M_n$.
\end{proof}}

Therefore the Hasse diagram for $\A$ has the form shown in Figure \ref{subfig2}. Note that for $j\in [0,k]$, we still have $a\meet j=j$ and $a\bdot j=a$ by the argument in Claim 2. Further, for any $\ell\in [k+1, n-1]$, we have $a<\ell$ which means $a\bdot\ell=0$ and $\ell\meet a=a$.

We know $a\meet \ell\in [0,a]$, but $a\meet \ell=a$ would contradict the minimality of $\A$, so in fact $a\meet \ell\in [0,k]$. Note that this also implies that $\ell\bdot a\neq a$. Since $\ell\bdot a\in [0, \ell]$, if $\ell\bdot a=i\in [0,k]\cup[k+1,\ell-1]$ then 
\[a\meet \ell=\ell\bdot i=\ell\notin [0,k]\,,\] a contradiction.  Thus, we must have $\ell\bdot a=\ell$ for all $\ell>a$. 

Putting these observations together, the updated Cayley table for $\A$ is show in Table \ref{tab:updated alg A}. The isomorphism $\A\to\M_{n+1}$ is given by
\begin{align*}
j&\mapsto j \text{ for $j\in [0,k]$}\\
a&\mapsto k+1\\
j&\mapsto j+1\text{ for $j\in [k+1, n-1]$}\,.
\end{align*}
\end{proof}

\begin{table}[h]
\centering
\begin{tabular}{c||cccccccccc}
$\bdot$  & 0 & 1 & 2 & $\cdots$ & $k$ & $a$ & $k+1$ & $\cdots$ & $n-2$ & $n-1$\\\hline\hline
0 & 0 & 0 & 0 & $\cdots$ & 0 & 0 & 0 & $\cdots$ & 0 & 0\\
1 & 1 & 0 & 0 & $\cdots$ & 0 & 0 & 0 & $\cdots$ & 0 & 0\\
2 & 2 & 2 & 0 & $\cdots$ & 0 & 0 & 0 & $\cdots$ & 0 & 0\\
$\vdots$ & $\vdots$ & $\vdots$ & $\vdots$ &  & $\vdots$ & $\vdots$ & $\vdots$ &  & $\vdots$ & $\vdots$\\
$k$ & $k$ & $k$ & $k$ & $\cdots$ & 0 & 0 & 0 & $\cdots$ & 0 & 0\\
$a$ & $a$ & $a$ & $a$ & $\cdots$ & $a$ & 0 & 0 & $\cdots$ & 0 & 0\\
$k+1$ & $k+1$ & $k+1$ & $k+1$ & $\cdots$ & $k+1$ & $k+1$ & 0 & $\cdots$ & 0 & 0\\
$\vdots$ & $\vdots$ & $\vdots$ & $\vdots$ &  & $\vdots$ & $\vdots$ & $\vdots$ &  & $\vdots$ & $\vdots$\\
$n-2$ & $n-2$ & $n-2$ & $n-2$ & $\cdots$ & $n-2$ & $n-2$ & $n-2$ & $\cdots$ & 0 & 0\\
$n-1$ & $n-1$ & $n-1$ & $n-1$ & $\cdots$ & $n-1$ & $n-1$ & $n-1$ & $\cdots$ & $n-1$ & 0\\
\end{tabular}
\caption{\label{tab:updated alg A} The updated Cayley table for $\A$}
\end{table}

Now, given a non-commutative BCK-algebra $\A$ of order $n$, the possible commuting degrees are 
\[\mc{CD}(n):=\Bigl\{\,\frac{3n-2}{n^2}\,,\, \frac{3n}{n^2}\,,\, \frac{3n+2}{n^2}\,,\, \ldots\,,\, \frac{n^2-2}{n^2}\Bigr\}\,,\] where the numerator increases by 2 since the relation $C(\A)$ is symmetric. Notice that $\mc{CD}(3)=\{\tfrac{7}{9}\}$.

An easy computation shows that \[|\mc{CD}(n)|=\frac{(n^2-2)-(3n+2)}{2}+1=\frac{(n-2)(n-1)}{2}=T_{n-2}\,,\] the $(n-2)^{\text{nd}}$ triangular number, and this means we can rewrite $\mc{CD}(n)$ as \[\Bigl\{\,\frac{n^2-2}{n^2}\,,\, \frac{n^2-4}{n^2}\,,\, \ldots\,,\, \frac{n^2-2T_{n-2}}{n^2}\Bigr\}\,.\] In \cite{evans23}, it was observed that all values in $\mc{CD}(3)$, $\mc{CD}(4)$, and $\mc{CD}(5)$ are obtained by some algebra of that order, and it was conjectured that this is true for all $n\geq 3$. 

At that time, this observation was made by simply computing the commuting degrees using the appendix in \cite{mj94}, which has Cayley tables for all BCK-algebras of orders 2, 3, 4, and 5. However, using a combination of Lemmas \ref{union with 2} and \ref{cd of sum}, we can construct algebras realizing all possible commuting degrees. Before we prove this result, we show some examples.

\begin{example}
Since $\2$ is commutative, we have $\cd(\2)=1=\frac{4}{4}$, and by Lemma \ref{cd of sum} we therefore have $\cd(\PI)=\cd(\2\oplus\top)=\frac{4+3}{3^2}=\frac{7}{9}$.

In the two tables below, we show the values of $\mc{CD}(4)$ and $\mc{CD}(5)$ as well as the algebras realizing them:
\begin{table}[h]
{\centering
\begin{tabular}{ccccc}
$\frac{10}{16}$ & & $\frac{12}{16}$  & & $\frac{14}{16}$ \\\\
$\PI\oplus \top$ & & $\TC\oplus\top$ & & $\PI\sqcup\2$\\
\end{tabular}
\caption{\label{CDs for n=4} Commuting degrees for $n=4$}
}
\end{table}

\begin{table}[h]
{\centering
\resizebox{\textwidth}{!}{\begin{tabular}{ccccccccccc}
$\frac{13}{25}$ & & $\frac{15}{25}$  & & $\frac{17}{25}$ & &
$\frac{19}{25}$ & & $\frac{21}{25}$  & & $\frac{23}{25}$\\\\
$(\PI\oplus \top)\oplus\top$ & & $(\TC\oplus\top)\oplus\top$ & & $(\PI\sqcup\2)\oplus\top$ & &
$(\PI\oplus \top)\sqcup\2$ & & $(\TC\oplus\top)\sqcup\2$ & & $(\PI\sqcup\2)\sqcup\2$\\
\end{tabular}}
\caption{\label{CDs for n=5} Commuting degrees for $n=5$}
}
\end{table}

\begin{figure}[h]
\centering
\begin{scaletikzpicturetowidth}{\textwidth}
\begin{tikzpicture}[scale=\tikzscale]
\node (A) at (0,0) {\small 7};
\node (B) at (-1, -1.5) {\small 10};
\node (C) at (0, -1.5) {\small 12};
\node (D) at (1, -1.5) {\small 14};

\draw [->,green] (A) -- (B);
\draw [->,red] (A) -- (D);

\node (E) at (-2.5, -3) {\small 13}; 
\node (F) at (-1.5, -3) {\small 15};
\node (G) at (-.5, -3) {\small 17};
\node (H) at (.5, -3) {\small 19};
\node (I) at (1.5, -3) {\small 21};
\node (J) at (2.5, -3) {\small 23};

\draw [->,green] (B) -- (E);
\draw [->,green] (C) -- (F);
\draw [->,green] (D) -- (G);
\draw [->,red] (B) -- (H);
\draw [->,red] (C) -- (I);
\draw [->,red] (D) -- (J);

\node (K) at (-4.5, -4.5) {\small 16};
\node (L) at (-3.5, -4.5) {\small 18};
\node (M) at (-2.5, -4.5) {\small 20};
\node (N) at (-1.5, -4.5) {\small 22};
\node (O) at (-.5, -4.5) {\small 24};
\node (P) at (.5, -4.5) {\small 26};
\node (Q) at (1.5, -4.5) {\small 28};
\node (R) at (2.5, -4.5) {\small 30};
\node (S) at (3.5, -4.5) {\small 32};
\node (T) at (4.5, -4.5) {\small 34};

\draw [->,green] (E) -- (K);
\draw [->,green] (F) -- (L);
\draw [->,green] (G) -- (M);
\draw [->,green] (H) -- (N);
\draw [->,green] (I) -- (O);
\draw [->,green] (J) -- (P);
\draw [->,red] (E) -- (O);
\draw [->,red] (F) -- (P);
\draw [->,red] (G) -- (Q);
\draw [->,red] (H) -- (R);
\draw [->,red] (I) -- (S);
\draw [->,red] (J) -- (T);

\node (U) at (-7, -6) {\small 19};
\node (V) at (-6, -6) {\small 21};
\node (W) at (-5, -6) {\small 23};
\node (X) at (-4, -6) {\small 25};
\node (Y) at (-3, -6) {\small 27};
\node (Z) at (-2, -6) {\small 29};
\node (AA) at (-1, -6) {\small 31};
\node (AB) at (0, -6) {\small 33};
\node (AC) at (1, -6) {\small 35};
\node (AD) at (2, -6) {\small 37};
\node (AE) at (3, -6) {\small 39};
\node (AF) at (4, -6) {\small 41};
\node (AG) at (5, -6) {\small 43};
\node (AH) at (6, -6) {\small 45};
\node (AI) at (7, -6) {\small 47};

\draw [->,green] (K) -- (U);
\draw [->,green] (L) -- (V);
\draw [->,green] (M) -- (W);
\draw [->,green] (N) -- (X);
\draw [->,green] (O) -- (Y);
\draw [->,green] (P) -- (Z);
\draw [->,green] (Q) -- (AA);
\draw [->,green] (R) -- (AB);
\draw [->,green] (S) -- (AC);
\draw [->,green] (T) -- (AD);
\draw [->,red] (K) -- (Z);
\draw [->,red] (L) -- (AA);
\draw [->,red] (M) -- (AB);
\draw [->,red] (N) -- (AC);
\draw [->,red] (O) -- (AD);
\draw [->,red] (P) -- (AE);
\draw [->,red] (Q) -- (AF);
\draw [->,red] (R) -- (AG);
\draw [->,red] (S) -- (AH);
\draw [->,red] (T) -- (AI);
\end{tikzpicture}
\end{scaletikzpicturetowidth}
\caption{Numerators appearing in $\mc{CD}(n)$ for $n=3,4,5,6,7$}
\label{fig:numerators}
\end{figure}
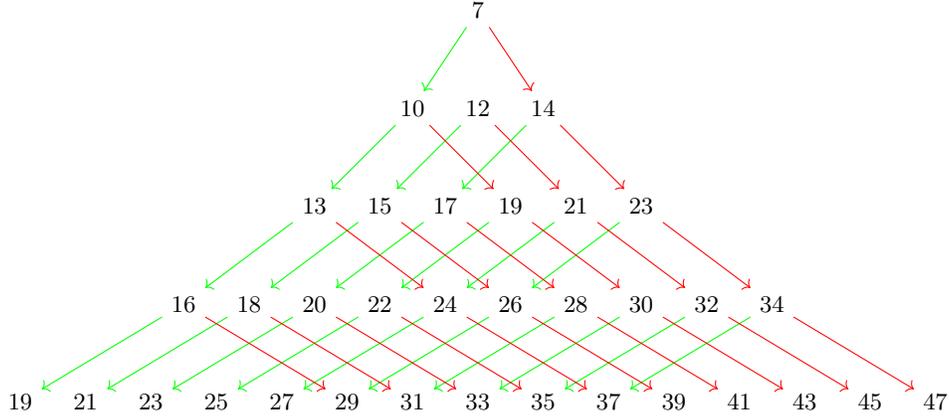
 
Starting with $n=3$, each level of Figure \ref{fig:numerators} shows the numerators appearing in $\mc{CD}(n)$. A green arrow indicates that the numerator below is achievable from the one above by an application of $-\oplus\top$, while a red arrow indicates the numerator below is achievable from the one above by an application of $-\sqcup\2$. Notice that some commuting degrees are achievable by both $-\oplus\top$ and $-\sqcup\2$. 
\end{example}

We prove below that every commuting degree in $\mc{CD}(k)$ is achieved by some algebra of order $k$. Looking at Figure \ref{fig:numerators}, the idea of the proof is this: at level $n$, there are $T_{n-2}$-many algebras, and we must construct $T_{n-1}$ algebras of order $n+1$ from them. Apply $-\oplus\top$ to all of the algebras on level $n$; this increases the numerators of their commuting degrees by 3. To construct the remaining $T_{n-1}-T_{n-2}=n-1$ algebras, we apply $-\sqcup\2$ to only the last $n-1$ algebras listed on level $n$, which increases the numerators by $2n+1$. Some tedious computation reveals that this will yield every commuting degree in the $n+1$ level. See also Figure \ref{fig:algebra path}.

\begin{theorem}\label{main result 1}
For every $n\geq 3$, every commuting degree in $\mc{CD}(n)$ is achieved by an algebra of order $n$.
\end{theorem}

\begin{proof}
We proceed by induction on $n$, noting again that the result holds for $\mc{CD}(3)$, $\mc{CD}(4)$, and $\mc{CD}(5)$. Suppose the result is true for $n\geq 6$ and set $t:=T_{n-2}$. Then there are $t$-many BCK-algebras of order $n$, call them $\A_1, \A_2, \ldots, \A_t$, such that $\mc{CD}(n)=\{\,\cd(\A_1), \cd(\A_2), \ldots, \cd(\A_t)\,\}$. Assume the algebras are labelled in order of increasing commuting degree. That is, assume
\[\cd(\A_1)=\frac{n^2-2t}{n^2}=\frac{3n-2}{n^2}\,,\,\cd(\A_2)=\frac{3n}{n^2}\,,\,\ldots\,,\, \cd(\A_t)=\frac{n^2-2}{n^2}\,.\]

Define a family of algebras as follows:
\begin{align*}
\B_1&=\A_1\oplus\top\\
\B_2&=\A_2\oplus\top\\
&\vdots\\
\B_t&=\A_t\oplus\top\\
\B_{t+1}&=\A_{t-(n-2)}\sqcup\2\\
\B_{t+2}&=\A_{t-(n-3)}\sqcup\2\\
&\vdots\\
\B_{t+(n-2)}&=\A_{t-1}\sqcup\2\\
\B_{t+(n-1)}&=\A_{t}\sqcup\2\,.
\end{align*} Notice that $t+(n-1)=T_{n-2}+(n-1)=T_{n-1}$, so we have $T_{n-1}$-many algebras all with order $n+1$. Applying Lemmas \ref{union with 2} and \ref{cd of sum} where appropriate, we compute the commuting degrees of the $\B_k$'s as
\begingroup
\allowdisplaybreaks
\begin{align*}
&\cd(\B_1)=\cd(\A_1\oplus\top)=\frac{(3n-2)+3}{(n+1)^2}=\frac{3(n+1)-2}{(n+1)^2}=\frac{(n+1)^2-2T_{n-1}}{(n+1)^2}\\\\
&\cd(\B_2)=\cd(\A_2\oplus\top)=\frac{3n+3}{(n+1)^2}=\frac{3(n+1)}{(n+1)^2}=\frac{(n+1)^2-2(T_{n-1}-1)}{(n+1)^2}\\\\
&\hspace{1.25cm}\vdots\\\\
&\cd(\B_t)=\cd(\A_t\oplus\top)=\frac{(n^2-2)+3}{(n+1)^2}=\frac{n^2+1}{(n+1)^2}=\frac{(n+1)^2-2\bigl(T_{n-1}-(t-1)\bigr)}{(n+1)^2}\\\\
&\begin{aligned}
\cd(\B_{t+1})=\cd(\A_{t-(n-2)}\sqcup\2)=\frac{\bigl(n^2-2(n-1)\bigr)+(2n+1)}{(n+1)^2}&=\frac{n^2+3}{(n+1)^2}\\\\
&=\frac{(n+1)^2-2\bigl(T_{n-1}-t\bigr)}{(n+1)^2}\end{aligned}\\\\
&\begin{aligned}
\cd(\B_{t+2})=\cd(\A_{t-(n-3)}\sqcup\2)=\frac{\bigl(n^2-2(n-2)\bigr)+(2n+1)}{(n+1)^2}&=\frac{n^2+5}{(n+1)^2}\\\\
&=\frac{(n+1)^2-2(T_{n-1}-(t+1)\bigr)}{(n+1)^2}\end{aligned}\\\\
&\hspace{1.25cm}\vdots\\\\
&\cd(\B_{t+(n-2)})=\cd(\A_{t-1}\sqcup\2)=\frac{(n^2-4)+(2n+1)}{(n+1)^2}=\frac{(n+1)^2-4}{(n+1)^2}\\\\
&\cd(\B_{t+(n-1)})=\cd(\A_{t}\sqcup\2)=\frac{(n^2-2)+(2n+1)}{(n+1)^2}=\frac{(n+1)^2-2}{(n+1)^2}\,.
\end{align*}
\endgroup
From this we see that every value in $\mc{CD}(n+1)$ is obtained. By induction, the theorem follows.

\end{proof}

In Figure \ref{fig:algebra path} we illustrate the construction in the proof going from $n=5$ to $n=6$ to $n=7$. As before, a green edge is an application of $-\oplus\top$ while a red edge is an application of $-\sqcup\2$. 

\begin{figure}[h]
\centering
\begin{scaletikzpicturetowidth}{\textwidth}
\begin{tikzpicture}[scale=\tikzscale]
\node (E) at (-2.5, -3) {\small $\A_1$}; 
\node (F) at (-1.5, -3) {\small $\A_2$};
\node (G) at (-.5, -3) {\small $\A_3$};
\node (H) at (.5, -3) {\small $\A_4$};
\node (I) at (1.5, -3) {\small $\A_5$};
\node (J) at (2.5, -3) {\small $\A_6$};

\node (K) at (-4.5, -4.5) {\small $\B_1$};
\node (L) at (-3.5, -4.5) {\small $\B_2$};
\node (M) at (-2.5, -4.5) {\small $\B_3$};
\node (N) at (-1.5, -4.5) {\small $\B_4$};
\node (O) at (-.5, -4.5) {\small $\B_5$};
\node (P) at (.5, -4.5) {\small $\B_6$};
\node (Q) at (1.5, -4.5) {\small $\B_7$};
\node (R) at (2.5, -4.5) {\small $\B_8$};
\node (S) at (3.5, -4.5) {\small $\B_9$};
\node (T) at (4.5, -4.5) {\small $\B_{10}$};

\draw [->,green] (E) -- (K);
\draw [->,green] (F) -- (L);
\draw [->,green] (G) -- (M);
\draw [->,green] (H) -- (N);
\draw [->,green] (I) -- (O);
\draw [->,green] (J) -- (P);
\draw [->,red] (G) -- (Q);
\draw [->,red] (H) -- (R);
\draw [->,red] (I) -- (S);
\draw [->,red] (J) -- (T);

\node (U) at (-7, -6) {\small $\C_1$};
\node (V) at (-6, -6) {\small $\C_2$};
\node (W) at (-5, -6) {\small $\C_3$};
\node (X) at (-4, -6) {\small $\C_4$};
\node (Y) at (-3, -6) {\small $\C_{5}$};
\node (Z) at (-2, -6) {\small $\C_{6}$};
\node (AA) at (-1, -6) {\small $\C_7$};
\node (AB) at (0, -6) {\small $\C_8$};
\node (AC) at (1, -6) {\small $\C_9$};
\node (AD) at (2, -6) {\small $\C_{10}$};
\node (AE) at (3, -6) {\small $\C_{11}$};
\node (AF) at (4, -6) {\small $\C_{12}$};
\node (AG) at (5, -6) {\small $\C_{13}$};
\node (AH) at (6, -6) {\small $\C_{14}$};
\node (AI) at (7, -6) {\small $\C_{15}$};

\draw [->,green] (K) -- (U);
\draw [->,green] (L) -- (V);
\draw [->,green] (M) -- (W);
\draw [->,green] (N) -- (X);
\draw [->,green] (O) -- (Y);
\draw [->,green] (P) -- (Z);
\draw [->,green] (Q) -- (AA);
\draw [->,green] (R) -- (AB);
\draw [->,green] (S) -- (AC);
\draw [->,green] (T) -- (AD);
\draw [->,red] (P) -- (AE);
\draw [->,red] (Q) -- (AF);
\draw [->,red] (R) -- (AG);
\draw [->,red] (S) -- (AH);
\draw [->,red] (T) -- (AI);

\end{tikzpicture}
\end{scaletikzpicturetowidth}
\caption{Obtaining $\mc{CD}(6)$ from $\mc{CD}(5)$, and $\mc{CD}(7)$ from $\mc{CD}(6)$.}
\label{fig:algebra path}
\end{figure}
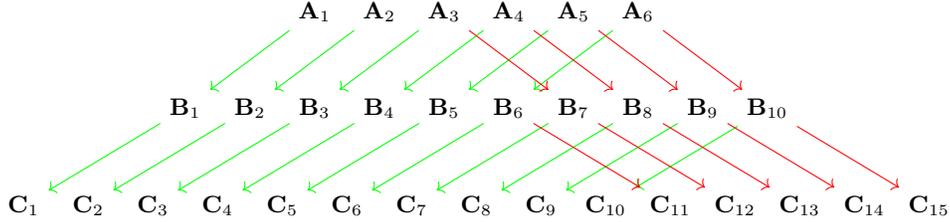

The procedure above is just one of very many ways to obtain commuting degrees. For example, the procedure above generates only 10 BCK-algebras of order 6, but in total there are 747 non-isomorphic non-commutative BCK-algebras of order 6. 

Finally, we show that every rational in $(0,1)$ is the commuting degree of some finite non-commutative BCK-algebra, and give an example of the procedure after the theorem.

\begin{theorem}\label{main result 2} 
Every rational in $(0,1)$ is the commuting degree for some finite non-commutative BCK-algebra. That is,
\[\bigcup\limits_{n=3}^\infty \mc{CD}(n)=\mathbb{Q}\cap (0,1)\,.\]
\end{theorem}

\begin{proof}
Take $\frac{p}{q}\in\mathbb{Q}\cap(0,1)$. Note that $1\leq p<q$. Set $n=2q$ and $k=2q(q-p)$. Notice that $k\geq 1$ by definition. On the other hand, 
\[T_{2q-2}=\frac{(2q-2)\bigl(2(q+1)-2\bigr)}{2}=2q(q-1)\geq 2q(q-p)=k\,.\] Thus, $k\in\{1,2,\ldots, T_{n-2}\}$. Finally, we observe that 
\[\frac{n^2-2k}{n^2}
=\frac{(2q)^2-2\bigl(2q(q-p)\bigr)}{(2q)^2}=\frac{4q^2-4q^2+4pq}{4q^2}=\frac{p}{q}\]
and since $k\in\{1,\ldots, T_{n-2}\}$, we have $\frac{p}{q}\in \mc{CD}(n)$. 
\end{proof}

Notice in fact that this shows the slightly stronger result that $\bigcup\limits_{\ell\geq 2}\mc{CD}(2\ell)=\mathbb{Q}\cap(0,1)$. We close on an example.

\begin{example}
We find some BCK-algebras with commuting degree $\frac{2}{5}$. By the proof of Theorem \ref{main result 2}, we should consider algebras of order $n=2q=10$, and the appropriate value of $k$ is 30. The procedure outlined in the proof of Theorem \ref{main result 1} gives us three algebras with commuting degree $\frac{2}{5}$:
\begin{align*}
((((((&\PI\oplus\top)\oplus\top)\oplus\top)\oplus\top)\sqcup\2)\oplus\top)\oplus\top\\
((((((&\TC\oplus\top)\oplus\top)\oplus\top)\sqcup\2)\oplus\top)\oplus\top)\oplus\top\\
((((((&\PI\sqcup\2)\oplus\top)\sqcup\2)\oplus\top)\oplus\top)\oplus\top)\oplus\top
\end{align*}
Figure \ref{fig:paths for 2/5} below shows the paths through Figure \ref{fig:numerators}, starting from either $\PI$ or $\TC\oplus\top$, to obtain these algebras.
\begin{figure}[h]
\centering
\begin{tikzpicture}
\node (A) at (0,0) {\small 7/9 $(\PI)$};
\node (B) at (-2.5, -1) {\small 10/16 };
\node (C) at (0, -1) {\small 12/16 $(\TC\oplus\top)$};
\node (D) at (2.5, -1) {\small 14/16};

\draw [->,green] (A) -- (B);
\draw [->,red] (A) -- (D);

\node (E) at (-2.5, -2) {\small 13/25}; 
\node (F) at (-0, -2) {\small 15/25};
\node (G) at (2.5, -2) {\small 17/25};

\draw [->,green] (B) -- (E);
\draw [->,green] (C) -- (F);
\draw [->,green] (D) -- (G);

\node (H) at (-2.5, -3) {\small 16/36};
\node (I) at (-0, -3) {\small 18/36};
\node (J) at (2.5, -3) {\small 28/36};

\draw [->,green] (E) -- (H);
\draw [->,green] (F) -- (I);
\draw [->,red] (G) -- (J);

\node (K) at (-2.5,-4) {\small 19/49};
\node (L) at (1.25, -4) {\small 31/49};

\draw [->, green] (H) -- (K);
\draw [->, red] (I) -- (L);
\draw [->, green] (J) -- (L);

\node (M) at (0, -5) {\small 34/64};
\node (N) at (0, -6) {\small 37/81};
\node (O) at (0, -7) {\small 40/100};

\draw [->,red] (K) -- (M);
\draw [->,green] (L) -- (M);
\draw [->,green] (M) -- (N);
\draw [->,green] (N) -- (O);
\end{tikzpicture}
\caption{Paths to obtain commuting degree $2/5$}
\label{fig:paths for 2/5}
\end{figure}
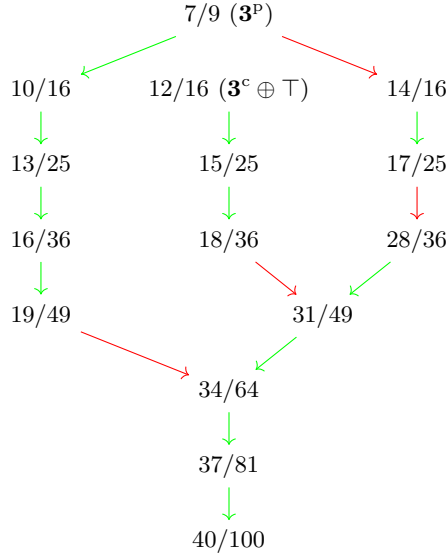

\end{example}

\bibliographystyle{plain}
\bibliography{comm_deg_bib}

\end{document}